\documentclass
{article}
 \usepackage{geometry}

\usepackage[center]{titlesec}

\usepackage[french]{babel}
\usepackage[T1]{fontenc}
\usepackage[latin1]{inputenc}
\usepackage{amsmath}
\usepackage{amsfonts}
\usepackage{amsthm}
\usepackage{bbold}
\usepackage{fancybox}
\usepackage{tikz}

\usepackage{mathdots}

\usepackage{setspace}

\usepackage{titlesec}
\titleformat{\subsection}[runin]
{\normalfont\large\bfseries}{\thesubsection}{1em}{}

 \usepackage{tikz-cd}

\DeclareMathOperator{\Div}{Div}
\DeclareMathOperator{\supp}{supp}
\DeclareMathOperator{\Pic}{Pic}
\DeclareMathOperator{\Princ}{Princ}

\DeclareMathOperator{\Jac}{Jac}

\DeclareMathOperator{\PGCD }{PGCD }
\DeclareMathOperator{\Ker}{Ker}

\usepackage{amssymb}
\usepackage[all]{xy}
\usepackage[nottoc, notlof, notlot]{tocbibind}
\usepackage{pb-diagram}
\usepackage{pb-xy}
\usepackage{amssymb} 
\usepackage{graphicx} 
\usepackage{blkarray}
\usepackage{titling,lipsum}

\newtheorem{theo}{Théorème }[section]
\newtheorem{lem}{Lemme}[section]
\newtheorem{prop}{Proposition}[section]

\newtheorem{rem}{Remarque}[section]
\newtheorem{defi}{Definition}[section]
\newtheorem{nota}{Notation}[section]
\newtheorem{exemple}{Exemple}[section]
\begin{document} 
\begin{center}
\section*{ Description de jacobiennes généralisées de courbes hyperelliptiques singulières à travers des espaces de phases de systèmes de Mumford }

\bigskip

Yasmine FITTOUHI\\

\bigskip
\end{center}
\section{Introduction}
De nombreux systèmes intégrables à dimension finie peuvent être exprimés à l'aide de l'équation de Lax qui met en évidence un paramètre spectral et par conséquent une courbe spectrale. Ces courbes spectrales sont le point de départ d'une investigation algébro-géométrique; l'aspect algébrique de cette investigation sera adjointe aux jacobiennes de courbes spectrales lisses et aux jacobiennes généralisées de courbes spectrales singulières et l'aspect géométrique sera attribué aux champs de vecteurs qui sont définis par l'équation de Lax. Cet article sera dédié uniquement aux courbes spectrales hyperelliptiques singulières qui sont associées à des systèmes de Mumford, on s'intéressera à la complémentarité de ces deux approches mathématiques pour décrire un même objet et dévoiler leurs caractéristiques.

\subsection{}
En 1954, Rosenlicht introduisait les jacobiennes généralisées de courbes singulières \cite{RM}.
Soit $C'$ une courbe lisse compacte et soient $p_1,\cdots, p_s$ des points distincts de $C'$ composant le support du diviseur effectif $\mathfrak{m}=\sum\limits_{i=1}^{k} n_ip_i$. Au moyen de la courbe lisse $C'$ et du diviseur $\mathfrak{m}$ on peut construire une courbe singulière $C$ en un point $s\in C'$ définie en deux étapes, on retire de la courbe $C'$ le sous ensemble $\{p_1,\cdots, p_k\}$,
$$C_{reg}=C'-\{p_1,\cdots, p_s\},$$
puis on attribut un point $s$ représentant le support de $\mathfrak{m}$
$$C=C_{reg}\cup \{s\}.$$ 
On obtient une projection $C'\longrightarrow C$.
\begin{defi}
Soit $\mathcal{O}'$ le faisceau de $C'$. On note par $\mathcal{O}$ un faisceau de $C$ défini de la manière suivante:
$$\mathcal{O}_p=\left\{ \begin{array}{cc} \mathcal{O}'_p &\text{ si } p\in C_{reg}, \\ \mathbb{C}+I_s &\text{ si } p=s,\end{array}\right.$$
où $I_s$ est l'idéal composé des fonctions $f$ de $\mathcal{O}'_s$ tel que pour tout $1\leqslant i \leqslant k$, les fonctions $f$ sont nulles au point $p_i$ au moins d'ordre $n_i$.
 
\end{defi} 
\begin{nota}
Soit $D_1$ et $D_2$ deux diviseurs de $C'$ étrangers au support de diviseur $\mathfrak{m}$. On note $D_1\underset{\mathfrak{m}}{\sim}D_2$, s'il existe une fonction méromorphe de $C'$ telle que $(f)=D_1-D_2$ et $f-1$ est nulle au points $\{p_i\}_{1\leqslant i \leqslant k}$ au moins d'ordre $\{n_i\}_{1\leqslant i \leqslant k}$ (respectivement). On note par $\Pic(C)$ l'ensemble des classes d'équivalence $\Div(C)/\underset{\mathfrak{m}}{\sim}$
\end{nota}
\begin{defi}
La jacobienne generalisée de la courbe $C$ est le groupe algébrique $\Pic^0(C)$ composé des diviseurs $D\in \Pic(C)$ de degré zéro; notée $\Jac(C)$ ou $\Jac_{\mathfrak{m}}(C')$.

\end{defi}
La construction de la jacobienne géneralisée nous permet de faire la liaison entre $\Jac(C)$ et $\Jac(C')$ avec la proposition ci-dessous:
\begin{prop}
La suite suivante
\begin{align}0\longrightarrow \mathbb{C}^{*k-1}\times\mathbb{C}^{(\sum\limits_{i=1}^kn_i-k)}\longrightarrow\Jac(C)\longrightarrow\Jac(C')\longrightarrow 0,\end{align}
est exacte.
\end{prop}

 \subsection{}
 Une courbe spectrale est une courbe d'équation affine $y^r +s_1(x)y^{r-1} +\cdots+s_r(x)=0$. Les courbes spectrales tiennent leur nom du fait qu'elles représentent le spectre de matrices et cela se traduit de la manière suivante:
 $$ \det(A(x)-Id_ry)=0,$$
 où $A(x)$ est une matrice carrée $r\times r$ de la forme
 $$A(x) = A_dx^d +A_{d-1}x^{d-1} +...+A_0 \text{ où } A_i\in gl(\mathbb{C}).$$
Les courbes hyperelliptiques sont des cas particuliers des courbes spectrales représentant le spectre de matrices carrées $2\times 2$. Nous allons focaliser notre étude aux courbes hyperelliptiques singulières associées aux systèmes de Mumford. Tout l'enjeu de cette article est de tirer profit des systèmes de Mumford pour compléter nos connaissances des jacobiennes généralisées de courbes hyperelliptiques.

\subsection{}
L'espace de phase $M_g$ du système de Mumford d'ordre $g$ avec $g\in \mathbb{N}^*$ est l'ensemble des matrices polynomiales carrées $2\times 2$ à trace nulle de la forme suivante: 
\begin{align}A(x) = A_{g+1}x^{g+1} +A_{g}x^{g} +...+A_0 \text{ avec } A_i\in gl_2(\mathbb{C}),\end{align}
 où $ A_{g+1}= \left(\begin{array}{cc}0&0\\1&0\end{array}\right)$ , $ A_{g}= \left(\begin{array}{cc}0&1\\w_g&0\end{array}\right)$ avec $w_g\in \mathbb{C}$.
 
Soit $A(x)=\left(\begin{array}{cc}v(x)&u(x)\\w(x)&-v(x)\end{array}\right)$ de $M_g$. L'équation de Lax du système de Mumford est la suivante:
\begin{align}
\frac{dA(x)}{dt}&= \left[A(x),-\displaystyle{\frac{A(t)}{x-t}}- \left( \begin{array}{cc}
0 & 0 \\ 
u(t) & 0
\end{array} \right)\right]. \label{c...}
\end{align}
Pour plus de details sur les systèmes de Mumford voir \cite{Mum}.

Soit $h$ un polynôme de degré $2g+1$. Soit $C$ la courbe hyperelliptique d'equation affine $y^2=h(x)$, on note $M_g(h)$ la fibre du système de Mumford associée à $C$.
Pour décrire les fibres des systèmes de Mumford, nous allons utiliser le fait que chaque fibre admet une stratification. La description des fibres revient à décrire les strates des systèmes de Mumford, comme on l'a vu précédemment chaque strate est isomorphe à une strate maximale d'un système de Mumford. Par conséquent notre description se réduit à décrire les strates
maximales, c'est-à-dire les strates $M_{g,g}(h)$ composée des matrices $A(x)=\left(\begin{array}{cc}v(x)&u(x)\\w(x)&-v(x)\end{array}\right)\in M_g(h)$ telles que $\deg(\PGCD(u,v,w))=0$. Rappelons, quand $h$ est sans racine multiple, $M_{g,g}(h)$ coincide avec la fibre $M_g(h)$ et est d'après Mumford un ouvert de la jacobienne de la courbe hyperelliptique lisse $C:y^2=h(x)$.
\subsection{}
Cet article est composé de trois parties. La première partie consiste à rappeler la construction de la jacobienne d'une courbe lisse ainsi que l'élaboration de la jacobienne généralisée d'une courbe singulière. Dans la deuxième partie nous nous restreignons aux courbes hyperelliptiques en rappelant leurs caractéristiques, puis nous définirons la jacobienne généralisée d'une courbe hyperelliptique singulière en calculant explicitement l'entier naturel $\delta$ par le lemme \ref{delta} qui est l'indice central du théorème de Riemann-Roch généralisé. Nous établirons aussi la jacobienne généralisée d'une courbe hyperelliptique singulière comme une extension de groupe à l'aide d'une suite exacte (\ref{suite}).Dans la troisième partie, après un bref rappel des systèmes de Mumford
comme systèmes algébriquement intégrables et de la stratifications des fibres de l'application
moment d'un système de Mumford. Comme chaque strate fine
d'une fibre singulière de $M_{g}$ est isomorphe 
à une strate maximale d'une fibre de l'espace de phase $M_{g'}$, où $g'<g$. Le reste de cette article sera dédié à établir toutes les caractéristiques des matrices composant une strate maximale dont la caractéristique la plus considerable est exprimée dans le lemme \ref{Deg}. Ceci nous permettra de construire 
 l'application injective $\Phi$ introduite dans la definition \ref{phi} qui liera la strate maximale $M_{g,g}(h)$ avec la jacobienne généralisée de la courbe hyperelliptique $C$ d'équation $y^2=h(x)$, puis nous nous emploierons à démontrer que $\Phi$ est une application bien définie et qu'elle est le bon choix pour définir un isomorphisme entre $M_{g,g}(h)$ et un ouvert de la jacobienne généralisée $\Jac_{\mathfrak{m}}(C')$, qui sera établi par la proposition \ref{fin}.

\section{Rappels}\label{Al}
Dans cette section, nous rappelons brièvement la notion de diviseur sur une courbe algébrique et nous énonçons la
définition et les propriétés fondamentales des courbes hyperelliptiques et nous rappelons la construction de la
jacobienne généralisée ( voir les details et informations complémentaires dans 
 \cite{GH} et \cite{Ser} ).

\bigskip

Soit $ {C}$ une courbe algébrique projective. Le \emph{groupe des diviseurs} de $ {C}$ est le groupe abélien
libre, engendré par les points de $ {C}$. Il est noté $\Div( {C})$ et ses éléments sont appelés
\emph{diviseurs} (de~$
{C}$). Tout diviseur $D$ de $ 
{C}$ s'écrit comme
$$
D=\sum\limits_{p\in 
{C}} n_p p, \text{ où les } n_p \text{ sont des entiers presque tous nuls}.
$$
Le \emph{support} de $D$, noté $\supp(D)$, est l'ensemble de points $p\in 
{C}$ tel que $n_p\neq 0$. Le
\emph{degré} de~$D$ est l'entier $\sum\limits_{p\in 
{C}} n_p $, noté $\deg(D)$. La fonction $\deg:\Div(
C)\to \mathbb Z$ est un morphisme de groupes, dont le noyau est noté $\Div^0( {C})$. Soit $D=\sum\limits_{p\in {C}} n_p p$ un diviseur de $%
{C} $ et $p$ un point de $ {C}$ on appelle
$n_p$ la \emph{multiplicité de $D$ en $p$}, noté $D\vert_{p}$.

\smallskip

Le groupe des diviseurs $\Div(
{C})$ est naturellement muni d'une relation d'ordre partiel. Soient
$D=\sum\limits_{p\in {C}} n_p p$ et $D'=\sum\limits_{p\in 
{C}} n'_p p$ deux diviseurs de $ {C}$. On dit que
$D\geqslant D'$ si $n_p \geqslant n'_p$ pour tout $p\in {C}$. Un diviseur $D \in \Div( {C})$ est dit
\emph{effectif} si $D \geqslant 0$. L'ensemble des diviseurs effectifs est noté $\Div^+( {C})$. Les diviseurs
effectifs de $ {C}$ de degré $n$ sont en bijection avec les éléments de $ {C}^{(n)}$, où $ {C}^{(n)}$ est le
produit symétrique $n$-ième de la courbe $ {C}$. Par abus de language, on appelle les éléments de $ {C}^{(n)}$
aussi des diviseurs de $C$.

\smallskip

Supposons maintenant que $ {C}$ soit une courbe non-singulière. Une fonction rationnelle $f$ sur $ {C}$
définit un diviseur, noté $(f)$, où
$$
 (f) = \sum\limits_{p\in {C}} v_p (f)p,
$$
avec $v_p$ la valuation de la fonction $f$ au point lisse $p$. Un tel diviseur est appelé diviseur
\emph{principal}. Les diviseurs principaux de $ {C}$ forment un sous-groupe de $\Div^0( {C})$, noté
$\Princ( {C})$. deux diviseurss $D$ et $D'$ de $ {C}$ sont dit \emph{linéairement équivalents} si leur
différence est un diviseur principal. On écrira alors $D\sim D'$ et on note $[D]$ la classe d'équivalence de
$D$. Les diviseurs principaux de $ {C}$ sont donc les diviseurs $D$ avec $D\sim 0$. La relation d'équivalence
linéaire est compatible avec la structure de groupe de $\Div( {C})$: si $D\sim D'$ alors $D+D''\sim D'+D''$ pour
tout $D''\in\Div( {C})$. Par conséquent, le quotient $ \Div( {C})/ \Princ( {C})$ est un groupe qu'on appelle
le \emph{groupe de Picard}, noté $\Pic( {C})$. Puisque le degré de tout diviseur principal est zéro, le morphisme
$\deg:\Div( {C})\to\mathbb{Z}$ induit un morphisme $\deg: \Pic(C) \longrightarrow \mathbb{Z}$. Son noyau est
constitué de classe d'équivalence linéaire de tous les diviseurs de degré zéro de $ {C}$ est un groupe abélien,
appelé la \emph{jacobienne} de la courbe $ {C}$, notée $\Jac( {C})$. C'est une variété abélienne de dimension
$g$ qui est très utile dans l'étude de la courbe lisse $ {C}$. Plus loin, nous considèrerons les jacobiennes
généralisées, qui généralisent la notion de jacobienne pour des courbes singulières, mais au préalable donnons une autre definition équivalente à cette dernière definition de la jacobienne d'une courbe lisse $ {C}$. 
\smallskip
\begin{defi}\upshape

Un faisceau de droites $L$ sur une surface de Riemann $C$, est une variété complexe de dimension deux, munie d'un holomorphisme $\pi: L \longrightarrow C$ tel que:

\begin{itemize}
\item L'image inverse par $\pi$ d'un point $p\in C$ a une structure d'espace vectoriel de dimension 1 noté $L_p$,
\item Pour tout $p\in C$, il existe un ouvert $U$ contenant $p$ est un isomorphisme $\phi_U: \pi^{-1}(U) \overset{\sim} {\longrightarrow}U\times \mathbb{C}$,
\begin{equation}
\begin{tikzcd}
 \pi^{-1}(U) \arrow{rr}{\overset{\phi_U}{\sim}} \arrow[swap]{dr}{\pi} & & U\times \mathbb{C} \arrow{dl}{ } \\[10pt]
 & U 
\end{tikzcd}
\end{equation}
\item Soit $U$ et $V$ deux ouverts, le morphisme $\phi_V\circ\phi_U^{-1}$ est de la forme \begin{equation}\label{tra}\begin{array}{ccc}
U\times \mathbb{C}& \longrightarrow &V\times \mathbb{C}\\
(p,x)& \longrightarrow &(p,f(p)x)
\end{array}\end{equation} avec $f$ une fonction holomorphe non-nul.
\end{itemize}
\end{defi}
\begin{nota}
Le holomorphisme $f$ défini dans (\ref{tra}) est noté $g_{UV}$ et est appelé fonction de transition des fonctions de trivializations locales $\phi_U$ et $\phi_V$.
\end{nota}
\begin{defi}\upshape

Une section holomorphe d'un faisceau de droites $L$ d'une surface de Riemann $C$ est une fonction holomorphe $s:C \longrightarrow L$ telle que $\pi\circ s=id$.
\end{defi}
Le relation entre une section $s$ et une fonction trivialization locale $\phi_U$ est la suivante 
\begin{equation}\begin{array}{cccc}
\phi_U(s):&U & \longrightarrow &U\times \mathbb{C}\\
&p& \longrightarrow &(p,s_U(p))
\end{array}\end{equation}
Où $s_U$ est une fonction holomorphique définie sur l'ouvert $U$. Soit $V$ un autre ouvert de $C$, on a sur l'intersection $U\cap V$
$$s_U=g_{UV}s_V.$$ 
\begin{rem}
Soit $s$ et $t$ deux sections d'un faisceau de droites $L$, et $U,V$ deux ouverts de $C$, on a sur l'ouvert $U\cap V$
\begin{equation} s_U=g_{UV}s_V 
\quad\text{et}\quad
t_U=g_{UV}t_V;\end{equation}
ce qui implique que sur $U\cap V$ on ait 
\begin{equation} \label{frac} \frac{s_U}{t_U}= \frac{s_V}{t_V};\end{equation}
En regroupant toutes les fractions de la forme (\ref{frac}), on reconstruit une fonction méromorphe sur toute la courbe $C$.\\
\end{rem}
\begin{nota}
Pour tous $s$ et $t$ deux sections d'un faisceau de droites $L$ et pour tout point $p$ de $C$, on note
\begin{equation}\label{es} (s+t)(m):=s(m)+t(m) 
\quad\text{et}\quad
(\lambda s)(m)=\lambda s(m).\end{equation}
avec $\lambda\in \mathbb{C}.$\\
L'ensemble des sections d'un faisceau de droites $L$ muni des relations (\ref{es}) est un espace vectoriel noté $H^0(C,L)$.
\end{nota}
\begin{theo}
Soit $L$ un faisceau de droites de $C$, l'espace $H^0(C,L)$ est de dimension finie.
\end{theo}
L'exemple suivant nous permettra d'apercevoir la relation entre courbe et faisceau en droite
\begin{exemple}
Soit l'espace projectif $\mathbb{P}^1$ muni de sa carte usuelle $(U_0,z)$ et 
$(U_1,{z}_1)$. On note par $\mathcal{O}(n)$ le faisceau en droite doté de la fonction de transition $g_{01}(z)=z^n$ sur $U_0\cap U_1=\mathbb{C}^*$. Une section $(s_0,s_1)$ de $H^0(\mathbb{P}^1,\mathcal{O}(n))$ doit satisfaire la condition suivante 
\begin{equation} \label{eq=}
s_0(z)=z^ns_1(z_1),
\end{equation}
sur l'intersection $U_0\cap U_1$.
Rappelons que les sections $(s_0,s_1)$ sont holomophes de forme:
$
s_0(z) = \sum\limits_{i=0}^\infty a_iz^i, \quad s_1(z_1) = \sum\limits_{i=0}^\infty b_iz_1^i,
$
en outre sur l'ouvert $\mathbb{C}^*$ on a $z_1=\frac{1}{z}$ alors l'equation (\ref{eq=}) devient 
\begin{equation}\label{eqbis=}
\sum\limits_{i=0}^\infty a_iz^i =z^n\sum\limits_{i=0}^\infty b_iz^{-i} \quad z\in \mathbb{C}^*,
\end{equation}
On a abouti au fait que la section $s(z)= \sum\limits_{i=0}^n a_iz^i$, car en identifiant, les coefficients des deux parties de l'equation (\ref{eqbis=}) on a $a_i=b_i=0$ pour $i>n$ et $a_i=b_{n-i}$ pour $i\leqslant n$. On conclut que toute section de section de $H^0(\mathbb{P}^1,\mathcal{O}(n))$ doit être un polynôme d'ordre au plus $n$, de ce fait $\dim(H^0(\mathbb{P}^1,\mathcal{O}(n)))\leqslant n+1$.

\end{exemple}
\bigskip
Soit $C$ une surface de Riemann et soit la suite exacte suivante
\begin{equation}\label{s.e}0 \longrightarrow \mathbb{Z} \overset{\exp(2\pi \star)}{\longrightarrow} \mathcal{O} \longrightarrow \mathcal{O}^*\longrightarrow1\end{equation}
où $ \mathbb{Z}$ est le faisceau des fonctions constantes sur $C$ à valeur dans $\mathbb{Z}$, $ \mathcal{O}$ est le faisceau des fonctions holomorphiques sur $C$ et $ \mathcal{O}^*$ est le faisceau des fonctions holomorphiques non-nulles sur $C$.\\
À l'aide de la cohomologie de groupe la suite exacte (\ref{s.e}) nous induit une suite exacte longue suivante: 
\begin{align}\label{s.e.l} \begin{split}
&0 \longrightarrow \mathbb{Z} {\longrightarrow} \mathbb{C} \longrightarrow \mathbb{C}^*
\longrightarrow \\
 &H^1(C,\mathbb{Z})\longrightarrow H^1(C,\mathcal{O})\longrightarrow H^1(C,\mathcal{O}^*)\\
\longrightarrow &H^2(C,\mathbb{Z})\longrightarrow H^2(C,\mathcal{O})\longrightarrow H^2(C,\mathcal{O}^*)\longrightarrow...
\end{split}
\end{align}
Cette longue suite s'arrête car $C$ est une courbe compacte on a $H^2(C,\mathbb{C})=0$. \\
De plus les flèches suivantes ont les particularités suivantes:
\begin{align}
 \mathbb{C} \overset{surjectif}{\longrightarrow} \mathbb{C}^*
\overset{\text{application } 0}{\longrightarrow} H^1(C,\mathbb{Z})\overset{injectif}{\longrightarrow} H^1(C,\mathcal{O})\longrightarrow H^1(C,\mathcal{O}^*)\longrightarrow H^2(C,\mathbb{Z}) \longrightarrow 0.
\end{align}
Alors la suite (\ref{s.e.l}) devient 
\begin{align}\label{s.e.e}
{0}{\longrightarrow} \frac{ H^1(C,\mathcal{O})}{H^1(C,\mathbb{Z})}\longrightarrow H^1(C,\mathcal{O}^*)\longrightarrow H^2(C,\mathbb{Z}) \longrightarrow 0.
\end{align}
En outre, le groupe $H^2(C,\mathbb{Z}) $ est isomorphe au groupe $\mathbb{Z}$, car $C$ est une surfaces de Riemann compactes.
\begin{nota}
L'application $H^1(C,\mathcal{O}^*)\longrightarrow H^2(C,\mathbb{Z})$ de suite exacte (\ref{s.e.e}) est notée $c_1$. Soit $L$ un faisceau de droite de $H^1(C,\mathcal{O}^*)$, on appelle $c_1(L)$ la première classe de Chern de $L$ et peut aussi être notée $deg(L)$.
\end{nota}
\begin{defi}\upshape

Soit $C$ une surface de Riemann lisse, le groupe $ \frac{ H^1(C,\mathcal{O})}{H^1(C,\mathbb{Z})}$ est appelé la \emph{jacobienne} de la surface de Riemann $C$.
\end{defi}
\subsection{Jacobiennes généralisées}\label{par:jac_gen}
Les jacobiennes généralisées sont des groupes algébriques non-compacts. Nous construirons les jacobiennes généralisées en s'appuyant sur les travaux de Serre \cite{Ser} , puis nous restreindrons notre étude aux courbes hyperelliptiques singulières. \\

Commençons par établir les notations qui seront utilisées dans cette section:
\begin{nota}\label{notat}
Soit $C$ une courbe projective singulière, avec $S$ l'ensemble de ses points
singuliers. On note ${C'}$ la normalisée de courbe projective $C$, la courbe ${C'}$ est lisse de genre $g'$. Soit $\mathfrak{n}$ un diviseur
effectif de ${C'}$, appelé \emph{module} (de $ {C'}$) tel qu'il existe une projection $\phi$ entre la courbe $C'$ et la courbe $C$ où
\begin{align}\label{sin}\phi:C'-\supp{\mathfrak{n}}\longrightarrow C-S\end{align} est un isomorphisme birégulier.
\end{nota}
On note aussi par $\phi$ le prolongement de l'application (\ref{sin})
\begin{align}\label{sin}\phi:C'\longrightarrow C\end{align}

\begin{defi}\upshape

Deux diviseurs $D$ et $D'$ de ${C'}$ sont dits $\mathfrak{n}$-\emph{équivalents}, noté
 $D\underset{\mathfrak{n}}{\sim}D'$, s'il existe une fonction rationnelle $f$ de $C$ telle que $(f-1) \vert_p
 \geqslant \mathfrak{n} \vert_p$ pour tout $p\in \supp(\mathfrak{n})$ et $D-D'=(f)$.
\end{defi}
\begin{rem}
Deux diviseurs $D$ et $D'$ de ${C'}$ sont
 $D\underset{\mathfrak{n}}{\sim}D'$, induit que le diviseur $D-D'$ est de support étranger à $\supp(\mathfrak{n})$.
 \end{rem}
\begin{nota}
L'ensemble des diviseurs ${C'}$ étrangers à $ \supp(\mathfrak{n})$, sera noté $\Div_{\mathfrak{n}}(C')$.
\end{nota}
\begin{defi}\upshape

 Soit $D$ un diviseur de $\Div_{\mathfrak{n}}(C')$. L'espace vectoriel $L_{\mathfrak{n}}(D)$ est composé des
 fonctions $f$ telles que $(f)\geqslant -D$ et $f$ admet une même valeur $c\in \mathbb{C}$ sur tous les points du
 $\supp(\mathfrak{n})$ avec $(f-c)|_{p}\geqslant \mathfrak{n} |_{p}$ pour tout $p\in \supp{\mathfrak{n}}$. La dimension de $L_{\mathfrak{n}}(D)$ est
 notée $l_{\mathfrak{n}}(D)$. L'espace vectoriel $I_\mathfrak{n}(D)$ est l'ensemble des formes $\omega$ telles que $(\omega)
 \geqslant D-\mathfrak{n} $. La dimension de $I_{\mathfrak{n}}(D)$ est notée $i_{\mathfrak{n}}(D)$.
\end{defi}

\begin{nota}
Pour tout point $p$ de $C$ (resp. tout point $q$ de $C'$), on note par $\mathcal{O}_p$ (resp. $\mathcal{O}'_q$) l'anneau local de $C$ (resp. $C'$).
\end{nota}
Sous les mêmes transcriptions vues dans la notation \ref{notat}, on illustre le poids d'un point singulier de la manière suivante:

 Soit $Q\in S$ \\
 $\bullet$ Si $\phi^{-1}(Q)=Q$ alors $\delta_Q= \dim(\mathcal{O}'_Q/\mathcal{O}_Q)$,\\
 $\bullet$ Si $\phi^{-1}(Q)=\{Q_1,\cdots,Q_L\}$ alors $\delta_Q= \dim(\mathcal{O}'_{Q_1}/\bigcap\limits_{1\leqslant i\leqslant L}
 \mathcal{O}'_{Q_i}).$

Le théorème classique de Riemann-Roch se généralise de la façon suivante:
\begin{theo}[Riemann-Roch]\label{thm:RR}
 Soient ${C'}$ une courbe projective lisse de genre $g'$ et soit $\mathfrak{n}$ un diviseur effectif de ${C'}$. Si
 $D$ est un diviseur de ${C'}$, étranger à $\supp(\mathfrak{n})$, alors
 $$l_{\mathfrak{n}}(D)-i_{\mathfrak{n}}(D)=\deg(D)+1-\pi\,.$$
 où
 $$\left\{ \begin{array}{lc} \pi= g' & \text{ si }
 \mathfrak{n}=0\;, \\ \pi= g' +\delta & \text{ si } \mathfrak{n}\neq 0\;, \\ \end{array} \right.$$ où
 $\delta=\sum\limits_{Q\in S} \delta_Q$ quand $\mathfrak{n}$ est un diviseur générique. Rappelons la definition de $\delta_Q$:
 Soit $Q\in S$ et soit l'application $\phi:C'\longrightarrow C$ définie plus haut (\ref{sin}):\\
 $\bullet$ Si $\phi^{-1}(Q)=Q$ alors $\delta_Q= \dim(\mathcal{O}'_Q/\mathcal{O}_Q)$,\\
 $\bullet$ Si $\phi^{-1}(Q)=\{Q_1,\cdots,Q_L\}$ alors $\delta_Q= \dim(\mathcal{O}'_{Q_1}/\bigcap\limits_{1\leqslant i\leqslant L}
 \mathcal{O}'_{Q_i}).$
\end{theo}
%
%
%

\smallskip

Ci-dessous nous énonçons quelques résultats qui découlent du théorème Riemann-Roch généralisé qui nous permettront de construire la jacobienne généralisée (voir \cite{Ser}):
\begin{lem}\label{RRL}
 Soient ${C'}$ une courbe projective lisse de genre $g'$ et soit $\mathfrak{n}$ un module de ${C'}$. Soit $D$ un
 diviseur de ${C'}$ étranger à $\supp(\mathfrak{n})$ et soit $p$ un point générique de ${C'}$. Si
 $i_{\mathfrak{n}}(D)\geqslant1$, alors 
 \begin{equation*} 
 i_{\mathfrak{n}}(D+p)=i_{\mathfrak{n}}(D)-1\;.
 \end{equation*}%
\end{lem}
Une application répétée de ce lemme prouve le résultat suivant:
\begin{lem}\label{RRL3}
 Soit ${C'}$ une courbe projective lisse de genre $g'$ et soit $\mathfrak{n}$ un module de ${C'}$. Soit $D$ un
 diviseur de degré zéro de ${C'}$ étranger à $\supp(\mathfrak{n})$ et soient $M_1,\dots, M_\pi$ des points génériques de ${C'}$. Alors il
 existe un unique diviseur effectif $D'$ de ${C'}$ tel que
 \begin{equation*}
 D'\underset{\mathfrak{n}}{\sim} D+\sum_{i=1}^\pi M_i\;.
 \end{equation*}%
\end{lem}
Lemme \ref{RRL3} 
implique que si $M$ et $N$ sont
des points génériques de ${C'}^{(\pi)}$, alors le lemme implique qu'il existe un unique élément $R$ de~${C'}^{(\pi)}$, satisfaisant
\begin{align}\label{P_{0}}
R\underset{\mathfrak{n}}{\sim}& M+N-\pi P_0,
\end{align}
où $P_0$ est un point arbitraire de $
{C'}-\supp(\mathfrak{n})$.

\begin{defi}\upshape

Soit la loi de composition $\star$ définie de la manière suivante:
 \begin{equation}\begin{array}{cccl}
\star:&{C'}^{(\pi)}\times {C'}^{(\pi)}& \longrightarrow &{C'}^{(\pi)}\\
&(M,N)& \longrightarrow &M\star N=R
\end{array}\end{equation}
où
 \begin{align}\label{P^{1}}
R\underset{\mathfrak{n}}{\sim}& M+N-\pi P_0,
\end{align}
avec $M$ et $N$ 
des points génériques de ${C'}^{(\pi)}$ et $P_0$ est un point arbitraire de $
{C'}-\supp(\mathfrak{n})$.
\end{defi}

\begin{prop}
 La variété ${C'}^{(\pi)}$, munie de la loi de composition $\star$ est un groupe birationnel.
\end{prop}
Énonçons maintenant un résultat fondamental qui nous permettra la construction de jacobiennes généralisées 
\begin{prop}\label{pj}
 Tout groupe birationnel est birationnellement isomorphe à un (unique) groupe algébrique.
\end{prop}
La proposition \ref{pj} révèle qu'il existe un groupe algébrique birationnellement isomorphe à $( {C'}^{(\pi)},\star)$. Ce groupe algébrique est de dimension $\pi$ et est appelé la
\emph{ jacobienne généralisée} de $ {C'}$ relative au module $\mathfrak{n}$, notée $\Jac_{\mathfrak{n}}( {C'})$. 
\begin{nota}
On note par $\varphi$ le morphisme birationnellement isomorphe entre $( {C'}^{(\pi)},\star)$ et $\Jac_{\mathfrak{n}}( {C'})$. 
\end{nota}
La construction suivante nous permettra de prolonger le morphisme $\varphi$ du groupe $( {C'}^{(\pi)},\star)$ vers le groupe des diviseurs étrangers à $\mathfrak{n}$.
Fixons un point $P_0$ arbitraire de ${C'}-\supp(\mathfrak{n})$ et $M=\sum\limits_{i=1}^\pi M_i$ un point générique de ${C'}^{(\pi)}$. Soit $D$ un diviseur de $\Div_{\mathfrak{n}}(C')$, il existe un élément de $ {C'}^{(\pi)}$ tel que
\begin{equation}
N \underset{\mathfrak{n}}{\sim} D-\deg(D)P_0+M.
\end{equation}
On note par $\theta$ le morphisme de $\Div_{\mathfrak{n}}(C')$ vers $\Jac_{\mathfrak{n}}( {C'})$ défini de la manière suivante:
 \begin{equation}\label{re}\begin{array}{cccl}
\theta:&\Div_{\mathfrak{n}}(C')& \longrightarrow &\Jac_{\mathfrak{n}}( {C'})\\
&D& \longrightarrow &\varphi(N)-\varphi(M)
\end{array}\end{equation}

\begin{prop}
L'application $\theta$ est indépendante du choix du point $P_0$ et de l'élément $M$.
\end{prop}
L'application $\theta$, peut-être appelée le \emph{morphisme de Serre}.
Exposons la relation entre la jacobienne généralisée de $C'$ associée au diviseur ${\mathfrak{n}}$ et la jacobienne généralisée de $C'$ associée au diviseur ${\mathfrak{n}'}$ où:
$$\supp(\mathfrak{n}')\subseteq \supp(\mathfrak{n}) \text{ et } \mathfrak{n}' \leqslant \mathfrak{n}.$$
Le morphisme suivant permet d'établir la relation entre les jacobiennes:
 \begin{equation}\begin{array}{cccl}
&({C'}^{(\pi)},\star)& \longrightarrow &\Div_{\mathfrak{n'}}({C'})\\
&(M_1, \cdots,M_{\pi'})& \longrightarrow &\sum\limits_{i=0}^{\pi}M_i
\end{array}\end{equation}
Du fait que $({C'}^{(\pi)},\star)$ est rationnellement isomorphe à $\Jac(C')_{\mathfrak{n}}$ et à l'aide du morphisme de Serre entre $\Div_{\mathfrak{n'}}({C'})$ et $ \Jac(C')_{\mathfrak{n'}}$, on aboutit au morphisme suivant: 
 \begin{equation}\label{is}\begin{array}{cccl}
&\Jac(C')_{\mathfrak{n}}& \longrightarrow &\Jac(C')_{\mathfrak{n'}}
\end{array}\end{equation}
\begin{rem}
Lorsque le diviseur $\mathfrak{n'}$ est nul, le morphisme (\ref{is}) sera entre la jacobienne généralisée $\Jac(C')_{\mathfrak{n}}$ et la jacobienne usuelle $\Jac(C')$.
\end{rem}
\smallskip

\section{Jacobienne généralisée d'une courbe hyperelliptique}
Nous allons nous intéresser, aux jacobiennes généralisées de courbes hyperelliptiques singulières, afin de compléter notre compréhension des fibres des systèmes de Mumford. Nous nous limiterons dans cette section aux rappels vus au
paragraphe \ref{par:jac_gen} des courbes hyperelliptiques et nous expliciterons le lien entre la jacobienne généralisée et la jacobienne usuelle évoquée dans la section \ref{Al}. 

\smallskip

\subsection{Courbes hyperelliptiques}
\begin{defi}\upshape
Une courbe hyperelliptique $C$ est une surface de Riemann, d'équation affine ${y^2=P(x)}$ où $P\in \mathbb{C}[x]$ avec $\deg(P)\geqslant 3$. Le genre arithmétique $g$ de la courbe $C$ est égal à la partie entière de la fraction $\left[\frac{\deg(P)}{2}\right]$.
\end{defi}

On a deux types d'équation affine de courbe hyperelliptique $C$ de genre $g$:
\begin{align}\label{eq:hyp_aff_odd}
 y^2&=x^{2g+1}+c_{2g}x^{2g}+\cdots+c_1x+c_0\;,\\
 y^2&=\prod\limits_{i=1}^{2g+1}(x-a_i);
\end{align}%
et
\begin{align}\label{eq:hyp_aff_even}
 y^2&=x^{2g+2}+c_{2g+1}x^{2g+1}+\cdots+c_1x+c_0\;,\\
 y^2&=\prod\limits_{i=1}^{2g+2}(x-a_i).
\end{align}
où les coefficients $c_i$ et les racines $a_i$ appartiennent à $\mathbb C$.\\%
Autrement exprimé une courbe hyperelliptique est les zéros de polynômes homogènes de formes suivantes 
$$\begin{array}{ccclcccl}
&\mathbb{P}^1& \longrightarrow &{ \mathbb{C}}\,&\,&\mathbb{P}^1& \longrightarrow & {\mathbb{C}}\\
&(x,y,z)&\longmapsto & y^2z^{2g-1}-\sum\limits_{i=0}^{2g+1}c_ix^{2g+1-i}z^{i},\,&\, &(x,y,z)&\longmapsto & y^2z^{2g}-\sum\limits_{i=0}^{2g+2}c_ix^{2g+2-i}z^{i}.
\end{array}$$ 

Toute courbe hyperelliptique admet une application $\imath$ appelée \emph{l'involution hyperelliptique} définie comme il suit:
$$\begin{array}{cccc}
\imath:&C& \longrightarrow &C\\
&(x,y)&\longmapsto &(x,-y)
\end{array}$$ 

Il y a des points d'une courbe hyperelliptique qu'on appelle les points à l'infini qu'on note $\infty$ et $+\infty,-\infty$, ces points à l'infini se manifestent lorsque $x$ atteint le point infini:
 \begin{gather}
\tilde{y}^2=\tilde{x}\prod\limits_{i=1}^{2g+1}(1-a_i\tilde{x});\quad \text{ ou }\quad \tilde{y}^2=\prod\limits_{i=1}^{2g}(1-a_i\tilde{x}).
\end{gather} 
où $\tilde{x}=\frac{1}{x}$ et $\tilde{y}=\frac{y}{x^{g}}$. Quand $x$ est le point infini, le point $\tilde{x}$ est zéro, ce qui implique \\
 \begin{gather}
\tilde{x}=0:\; \quad \tilde{y}^2=0 \quad \text{ ou }\quad \tilde{y}^2=1.
\end{gather} 
Alors les points à l'infini de courbes hyperelliptiques sont
$$\begin{array}{clcccc}
 \infty &\text{ attribué à }& \tilde{x}=0,\;& \quad \tilde{y}=0 \quad &\text{ pour }&\quad C: y^2=\prod\limits_{i=1}^{2g+1}(x-a_i),\\
+\infty,-\infty & \text{ attribués à } & \tilde{x}=0,\;& \quad \tilde{y}=\pm1 \quad &\text{ pour }&\quad C: y^2=\prod\limits_{i=1}^{2g+2}(x-a_i).
\end{array}$$
\smallskip

Soit $C$ 
une courbe hyperelliptique singulière \footnote{ les notations introduites dans ce paragraphe seront utilisées tout au long de cet article. }, d'équation affine $y^2=h(x)$, où $h$ est un polynôme unitaire de
degré $2g+1$. Notons par $P(x)$ le diviseur quadratique maximal de $h(x)$, c'est-à-dire $P(x)$ est le polynôme
unitaire de degré maximal tel que $P^2(x)$ divise $h(x)$. La normalisée de $C$ est aussi une courbe hyperelliptique
lisse qu'on notera $C'$ d'équation affine $z^2=h'(x)$. Le genre de $C'$ est $g'$ car $h'$ est un polynôme unitaire de degré $2g'+1$.
On a le morphisme entre $C'$ et $C$ défini de la manière suivante:
$$
\begin{array}{rcccc}
 \phi&:&C' &\longrightarrow & C \\
 &&(x,z)& \longmapsto & (x,P(x)z). \\ 
\end{array}
$$

\smallskip

Afin de définir le module sur $C'$ qui correspond aux points singuliers de $C$, nous devons factoriser $P$ et distinguer les racines communes entre $P$ et $h'$ des autres.\\
La factorisation de $P$ est la suivante:
\begin{align}P(x)=\prod\limits_{i=1}^{k}(x-a_{i})^{\ell_{i}}\qquad \text{ avec } \{\ell_1,\cdots,\ell_k\}\in\mathbb{N}^*\end{align}
où toutes les racines $a_{i} \in \mathbb{C}$ de $P$ sont distinctes, avec $h'(a_{i})=0$ pour $1\leqslant i\leqslant d$ et
$h'(a_{i})=b_i^2\neq 0$ pour $d+1\leqslant i\leqslant k$. 
Notons par $n$ le degré de $P$, alors
\begin{equation*}
 n=\sum_{i=1}^k\ell_i\; \qquad \hbox{ et }\qquad g=g'+n\;.
\end{equation*}%
On note par $\mathfrak{m}$ le module de ${C'}$, défini par
\begin{align*}
 \mathfrak{m}&=\sum\limits_{i=1}^{d}2\ell_{i}(a_{i},0)+\sum\limits_{i=d+1}^{k}\ell_{i}((a_{i},b_{i})+(a_{i},-b_{i})).
\end{align*}
Dans le lemme suivant nous calculons la valeur des entiers $\delta$ et $\pi$ qui figurent dans le théorème de Riemann-Roch \ref{thm:RR}.
\begin{lem}\label{delta}
 L'entier $\delta$ est donné par $\delta=\sum\limits_{i=1}^k\delta_{(a_i,0)}$ où $\delta_{(a_i,0)}=\ell_i$, pour
 $i=1,\dots,k$. En particulier, $\pi=g'+\delta=g'+n=g$.
\end{lem}
\begin{proof}[Preuve]
Soit $\phi$ la projection de $C'$ vers $C$:
$$\begin{array}{cccc}\phi:&C'&\longrightarrow& C\\ &(x,z) &\longmapsto& (x,P(x)z)\end{array}\,$$ Pour un point $p$
de $C$ (resp.\ un point $q$ de $C'$), on note par $\mathcal{O}_{p}$ (resp. $\mathcal{O}'_{q}$) l'anneau local de
$C$ (resp. $C'$).
Pour $p=(a_{i},0)$ pour $1\leqslant i\leqslant d$
$$\mathcal{O}_{p}= \{g\in \mathbb{C}(C)\mid f=g \circ \phi \text{ avec } f\in\mathcal{O}'_{(a_{i},0)} \text{ et } v_{(a_{i},0)}(f)\geqslant \ell_{i}\}\;.$$

Par définition, $\delta_{(a_{i},0)}=\dim(\mathcal{O}'_{(a_{i},0)}/\phi^{-1}(\mathcal{O}_{(a_{i},0)}))$, donc
$$
 \delta_{(a_{i},0)}=\dim\left<1,y,y^2,\dots,y^{\ell_{i}-1}\right>=\ell_{i}\;.
$$
Pour $p=(a_{i},0)$ pour $d+1\leqslant i\leqslant k$
$$\mathcal{O}_{p}= \{g\in \mathbb{C}(C)\mid f=g \circ \phi \text{ avec } f\in\mathcal{O}'_{(a_{i},b_{i})}\cap
\mathcal{O}'_{(a_{i},-b_{i})} \text{ et } v_{(a_{i},b_{i})}(f)=v_{(a_{i},-b_{i})}(f)\geqslant \ell_{i}\}\;.$$
Par la définition de $\delta_{(a_{i},0)}$, on a dans ce cas également
$$
 \delta_{(a_{i},0)}=\dim\left<1,(x-a_{i}),(x-a_{i})^2,\dots,(x-a_{i})^{\ell_{i}-1}\right>=\ell_{i}\;.
$$
L'application $\phi$ est un isomorphisme entre $C-\{(a_{1},0),\dots,(a_{d},0),(a_{d+1},0),\dots, (a_k,0)\}$ et
${C'-\supp{\mathfrak{m}}} $ où on rappelle que
$\mathfrak{m}=\sum\limits_{i=1}^{d}2l_{i}(a_{i},0)+\sum\limits_{i=d+1}^{k}l_{i}((a_{i},b_{i})+(a_{i},-b_{i}))$.
Pour tout point $p$ de $ C- \left\{(a_{1},0),\dots,(a_{d},0),(a_{d+1},0),\dots, (a_k,0)\right\}$,
on a
$$\mathcal{O}_{p}= \{g\in \mathbb{C}(C)\mid f=g \circ \phi \text{ avec } f\in\mathcal{O}'_{\phi^{-1}(p)}\}\;.$$
Par conséquent, 
$$\delta_{p}= \dim \mathcal{O}'_{\phi^{-1}(p)}/\phi^{-1}(\mathcal{O}_{p})=0\; $$
pour tout $p$ différent des points singuliers de $C$.
Puisque, par définition, $\delta=\sum\limits_{p\in C}\delta_{p}$, on trouve 
$$\delta=\sum\limits_{i=1}^{k}\delta_{(a_{i},0)}=\sum_{i=1}^k\ell_{i}=n\;,$$
et donc $$\pi=g'+n=g\;.$$
\end{proof}
Le théorème Riemann-Roch (\ref{thm:RR}) prend la forme suivante 
$$l_{\mathfrak{m}}(D)-i_{\mathfrak{m}}(D)=\deg(D)+1-g\;.$$
\smallskip
D'après les rappels du paragraphe
\ref{par:jac_gen}, on sait que le produit symétrique ${C'}^{(g)}$ muni d'une loi de composition rationnelle en fait un groupe birationnel,
qui est birationnellement isomorphe avec la jacobienne généralisée
$\Jac_{\mathfrak m}({C'})$.
Le morphisme de Serre $\theta$ est entre le groupe des diviseurs de ${C'}$, étrangers à $\supp(\mathfrak{m})$, et
$\Jac_\mathfrak{m}({C'})$ induit un autre morphisme $\tau:\Jac_\mathfrak{m}({C'})\to \Jac_0({C'})=\Jac({C'})$., En traduisant 
un élément $u$ de $\Jac_\mathfrak{m}({C'})$ comme la classe de $\mathfrak{m}$-équivalence d'un diviseur de degré zéro, disjoint
du support de $\mathfrak{m}$, dès lors l'élément $\tau(u)$ est simplement la classe d'équivalence linéaire de $u$. Ceci implique que
$\Jac_\mathfrak{m}({C'})$ est une extension de $\Jac({C'})$ par $L_\mathfrak{m}:=\Ker\tau$. Serre décrit ce dernier groupe
abélien comme produit de plusieurs copies du groupe additif $\mathbb C$ et du groupe multiplicatif $\mathbb
C^*$. Sur la courbe hyperelliptique singulière $C$, on peut montrer que
$$
 L_\mathfrak{m}\simeq\mathbb{C}^{*k-d}\times \mathbb{C}^{n-k+d}\;.
$$
De ces faits, on a la suite exacte de
groupes abéliens suivante:
\begin{equation}\label{suite}
 0\longrightarrow \mathbb{C}^{*k-d}\times \mathbb{C}^{n-k+d}\longrightarrow \Jac_\mathfrak{m}({C'})\longrightarrow
 \Jac({C'})\longrightarrow 0\;.
\end{equation}%
\section{Description géométro-algébrique des strates $M_{g,g}(h)$}
L'objectif de cette section est de décrire les fibres du système de
Mumford impair d'ordre $g$ à l'aide des outils de la géométrie algébrique. Lorsque la fibre est lisse , elle est décrite par Mumford comme isomorphe à une jacobienne usuelle moins son
diviseur thêta. Consacrons-nous à l'étude des fibres singulières. 
\subsection{Rappel du systeme de Mumford}
Soit $g$ entier positif et soit $M_g$ un espace affine complexe de dimension ${3g+1}$ de coordonnées $ u_{g-1},u_{g-2}, \cdots ,u_0,$ $v_{g-2},\cdots,v_0,$ $w_{g},w_{g-1},\cdots,w_0$ , s'exprimant à l'aide de matrices polynomiales $2\times 2$ de trace nulle de la forme suivante:
$$M_g:=\left\lbrace \left( \begin{array}{cc}
v(x) & u(x) \\
w(x) & -v(x)
\end{array} \right) \text{ tel que }
\begin{array}{ccl}
 u(x) &=& x^g+u_{g-1}x^{g-1}+u_{g-2}x^{g-2}+\cdots+u_0\\
 v(x) & =& v_{g-1}x^{g-1}+v_{g-2}x^{g-2}+\cdots+v_0\\
 w(x) & = & x^{g+1}+w_{g}x^{g}+w_{g-1}x^{g-1}+\cdots+w_0
\end{array} \right\rbrace \simeq\mathbb{C}^{3g+1}. $$
Le système de Mumford d'ordre $g$ est un système intégrable d'espace de phases $M_g$.%

Le determinant d'une matrice $A(x)= \left( \begin{array}{cc}
v(x) & u(x) \\ 
w(x) & -v(x)
\end{array} \right)$ de $M_g$, est un polynôme de degré $2g+1$ 
$$\det(A(x))=-[v(x)^2+u(x)w(x)].$$
La forme des determinants des matrices de $M_g$, nous incite à définir l'application $\mathbf{H}$ 
$$\begin{array}{cccl}
 \mathbf{H}:&M_{g}& \longrightarrow & {H}_{g}\\
 &A& \longmapsto & -\det(A(x))
\end{array}$$
où ${H}_{g}$ est l'ensemble des polynômes unitaires de degré $2g+1$.
L'application $\mathbf{H}$ est l'application \emph{moment} du système de Mumford.

\smallskip

La fibre de l'application moment $\mathbf{H}$ au-dessus d'un polynôme $h(x)$ de ${H}_{g}$ est notée
$M_{g}(h)$ et est déterminée de cette manière 
$$M_{g}(h)=\left\{A(x)\in M_{g} \mid \det(y{I_{2}}-A(x))=y^{2}-h(x) \right\}.$$
Toutes les matrices appartenant à une fibre $M_g(h)$ ont un même polynôme caractéristique, les annulateurs de ce dernier 
forment une courbe hyperelliptique $C$ d'équation affine $y^2=h(x)$ et de genre arithmétique $g$. Lorsque $h(x)=P^2(x)h'(x)$, la fibre est associée à une courbe hyperelliptique singulière d'équation affine $C:y^2=h(x)$.
\smallskip
Nous avons vu précédemment, que quand le polynôme $h$ admet des racines multiples, la fibre $M_g(h)$ de $\mathbf H$ admet
une stratification (la stratification \emph{fine}) telle que chaque strate est isomorphe à la strate maximale d'une
fibre d'un système de Mumford d'ordre inférieur. Par conséquent, afin de décrire les fibres singulières $M_{g}(h)$
du système de Mumford, il suffit de décrire les strates maximales $M_{g,g}(h)$ où l'on rappelle que
$$M_{g,g}(h)=\left\{A(x)=\left( \begin{array}{cc}
 v(x)& u(x) \\ 
 w(x) & -v(x)
\end{array}\right)\in M_{g}(h)\mid \PGCD(u,v,w)=1 \right\}.
$$
On note par $(D_i)_{i=0,1,\dots,g-1}$ les champs de vecteurs qui définissent le système de Mumford et qui s'écrivent sous la forme de l'équation de Lax
\begin{align}
D_i\vert_A&=\left[A(x),\left[ \displaystyle{\frac{A(x)}{x^{i+1}}}\right]_+ - \left( \begin{array}{cc}
0 & 0 \\ 
u_i & 0
\end{array} \right)\right], \label{c'}
\end{align}
avec $\left[ {\frac{A(x)}{x^{i+1}}}\right]_+$ est la partie polynomiale de la matrice $ {\frac{A(x)}{x^{i+1}}}$. 
On peut réécrire $M_{g,g}(h)$ à l'aide des champs de vecteurs $(D_i)_{i=0,1,\dots,g-1}$ de la manière suivante:
$$M_{g,g}(h)=\left\{A(x)\in M_{g}(h)\mid \dim<D_0\vert_A,\cdots,D_{g-1}\vert_A> =g \right\}.
$$
\smallskip
\begin{nota}
Pour deux polynômes $S$ et $q$, on note par $S_{q}$ le quotient de $S$ par $q$,
$$S(x)=q(x)S_{q}(x).$$
L'ensemble des polynômes unitaires $Q$ de $\mathbb{C}[x]$ tel que $Q^2$ divise $h$ est noté par
$\mathbb{C}[x]_{h}$.
\end{nota}

\bigskip
\subsection{Lien entre la strate $M_{g,g}(h)$ et la jacobienne $\Jac_{\mathfrak{m}}(C')$.}
Dans cette section, nous définissons un morphisme injectif entre la strate $M_{g,g}(h)$ et un ouvert de la
jacobienne $\Jac_\mathfrak{m}({C'})$. Nous montrons également que les champs de vecteurs indépendants $D_i$ sont
envoyés par ce morphisme sur des champs invariants sur le groupe algébrique $J_\mathfrak{m}({C'})$.
\smallskip
Fixons le point $P_{0}$ égal au point $\infty$ de la courbe lisse ${C'}$, comme le point de base du
morphisme de Serre $\theta$ entre le groupe des diviseurs de ${C'}$, étrangers à $\supp(\mathfrak{m})$, et
$\Jac_\mathfrak{m}({C'})$. Son noyau est formé des diviseurs étrangers au support de $\mathfrak{m}$ qui sont
$\mathfrak{m}$-équivalents à un multiple de $\infty$.
\begin{lem}\label{Deg}
Soit $A(x)=\left( \begin{array}{cc}
 v(x)& u(x) \\ 
 w(x) & -v(x)
\end{array}\right)\in M_{g,g}(h)$.\\

Si $\PGCD(P^{2},u)=R$, alors $R$ est le carré d'un polynôme unitaire $Q=\PGCD(P,u,v)$. Réciproquement si $\PGCD(P,u,v)=Q$ alors $\PGCD(P^{2},u)=Q^{2}$.\\ 

\smallskip

Si $\PGCD(P^{2},w)=R$, alors $R$ est le carré d'un polynôme unitaire $Q=\PGCD(P,w,v)=Q$. Réciproquement si $\PGCD(P,w,v)=Q$ alors $\PGCD(P^{2},w)=Q^{2}$.
\end{lem}
\begin{proof}[Preuve]
Soit $A(x)=\left( \begin{array}{cc} v(x)& u(x) \\ w(x) & -v(x)
\end{array}\right)\in M_{g,g}(h)$.
Soit $a\in \mathbb{C}$ une racine de $\PGCD(P^{2},u)$, supposons que $\PGCD(P^{2},u,(x-a)^{2k})=(x-a)^{2k-1}$ avec $k\in \mathbb{N}^{*}$.
On rappelle que 
\begin{align} \label{eqa2}
P^{2}(x)h'-u(x)w(x)&=v^{2}(x).
\end{align}
Le polynôme $(x-a)^{2k-1}$ divise le côté gauche de l'égalité (\ref{eqa2}), donc $(x-a)^{2k-1}$ divise $v^{2}(x)$
ceci implique que $(x-a)^{2k}$ divise $v^{2}(x)$. De même, comme $(x-a)^{2k-1}$ divise $P^{2}(x)$ ceci implique que
$(x-a)^{2k}$ divise $P^{2}(x)$, donc le polynôme $(x-a)^{2k}$ divise $u(x)w(x)$, étant donné que 
$\PGCD(u,v,w)=1$ et comme $a$ est une racine de $u$ alors $(x-a)^{2k}$ divise $u(x)$. on a bien: $$\PGCD(P^{2},u,(x-a)^{2k})=(x-a)^{2k}\;,$$ 
ce qui contredit l'hypothèse.
On conclut que toute racine du polynôme $\PGCD(P^{2},u)$ est une racine d'ordre pair,
d'où l'existence d'un polynôme unitaire $Q$ tel que $Q^{2}(x)=\PGCD(P^{2},u)$.

\smallskip

Afin de montrer que $\PGCD(P,u,v)=Q$, il suffit de montrer que toutes les racines de $Q$ sont des racines du $\PGCD(P,u,v)$ avec la même multiplicité. L'égalité (\ref{eqa2}), nous informe que si $\PGCD(P^{2},u)=Q^{2}$ alors
$Q$ divise $v$. Soit $a$ une racine de $Q$ d'ordre $k$, c'est-à-dire $a$ est une racine de $\PGCD(P^{2},u)$
d'ordre $2k$. En d'autres termes, $a$ est une racine d'ordre au moins $k$ de $u$, $P$ et $v$. Si $a$ est une racine d'ordre au moins $k+1$ de $P$ et $v$. L'égalité (\ref{eqa2}) et le fait que $\PGCD(u,v,w)=1$ impliquent que $a$ est une racine $u$ d'ordre $2k+2$; ceci entraine que $a$ est une racine d'ordre $2k+2$ de $\PGCD(P^{2},u)$, ce qui est une contradiction. Donc $a$ est une racine de $Q$ d'ordre $k$ et $\PGCD(P,u,v)=Q$.

\smallskip

Pour la réciproque, si $\PGCD(P,u,v)=Q$, l'égalité (\ref{eqa2}) devient
\begin{align}\label{eqa3}
Q^{2}(x)[P_{Q}^{2}(x)h'-v_{Q}^{2}(x)]&=u(x)w(x)\;.
\end{align}
De égalité (\ref{eqa3}) et du fait que $\PGCD(u,v,w)=1$, on a que $Q^{2}$ divise $u$, par conséquent $Q^{2}$ divise $\PGCD(P^{2},u)$. Si $Q^{2}$ divise $\PGCD(P^{2},u)$, mais $\PGCD(P^{2},u)\neq Q^{2}$, c'est-à-dire $\PGCD(P,u,v)\neq Q$, ceci est une contradiction. On conclut que si $\PGCD(P,u,v)=Q$ alors $\PGCD(P^{2},u)=Q^{2}$.

\smallskip

Pour prouver le second point de la proposition on refait les mêmes étapes et conclusions en remplaçant le polynôme $u$ par le polynôme $w$. 
\end{proof}
\begin{rem}
Soit $A(x)\in\left( \begin{array}{cc}v(x)&u(x)\\w(x)&-v(x)\end{array}\right)\in M_{g,g}(h)$, et soit $Q$ le
polynôme unitaire tel que $Q^2=\PGCD(P^{2},u)$, alors $Q\in \mathbb{C}[x]_{h}$.
\end{rem}
\smallskip

\begin{defi}\label{phi}\upshape
On note par $\Phi$ l'application entre $M_{g,g}(h)$ et $\Jac_{\mathfrak{m}}(C')$ définie de la manière suivante: 
$$\begin{array}{cccl}
 \Phi:&M_{g,g}(h)&\longrightarrow&\Jac_{\mathfrak{m}}(C')\\
 &\left(\begin{array}{cc}
 v(x)&u(x)\\w(x)&-v(x)\end{array}\right)
 &\longmapsto&\theta\left(\left(\frac{R(x)(P(x)z+v(x))}{u(x)}+1\right)_{0}\right)\;,
\end{array} $$
où $R(x)=\prod\limits_{i=1}^{k}(x-a_{i})^{\ell_{i}+1}$\;.
\end{defi}
\begin{prop}
L'application $\Phi$ de $M_{g,g}(h)$ vers $\Jac_{\mathfrak{m}}(C') $ est bien définie.
\end{prop}
\begin{proof}[Preuve]
 Pour montrer que l'application $\Phi$ est bien définie, il suffit de montrer que le diviseur des zéros de
 $F(x,z)=\frac{R(x)(P(x)z+v(x))}{u(x)}+1$ est étranger à $\supp(\mathfrak{m})$.

 \smallskip
 
Soit $A(x)=\left(\begin{array}{cc}v(x)&{u(x)}\\w(x)&-v(x)\end{array}\right)\in M_{g,g}(h)$, et soit $Q$ le polynôme égal au $\PGCD(P,u,v)$
de degré $j$. D'après le lemme \ref{Deg}, $Q^{2}$ divise le polynôme $u$, de plus 
$Q$ divise $R$, alors
$$F(x,z)=\frac{R_{Q}(x)(P_{Q}(x)z+v_{Q}(x))}{u_{Q^{2}}(x)}+1\;.$$
On note par $\{x_i\}_{1\leqslant i\leqslant g-2j}$ les racines du polynôme $u_{Q^{2}}(x)$:
$$u_{Q^2}(x)=\prod\limits_{i=1}^{g-2j}(x-x_{i}).$$ 
Le diviseur principal de la fonction
$F(x,z)$
est donné par
\begin{equation}\label{eqd}\left(F(x,z)\right)=-\sum\limits_{i=1}^{g-2j}\left(x_{i},\frac{v_{Q}(x_{i})}{P_{Q}(x_{i})}\right)-(2(n+k)+1)\infty+D\;,\end{equation}
où $D$ est un diviseur effectif de degré $g+2(n+k-j)+1$. Afin d'expliciter l'égalité (\ref{eqd}) , on détermine d'abord les zéros de $u_{Q^2}$ sur $C'$:
$$
 \left(x_{i},\frac{v_{Q}(x_{i})}{P_{Q}(x_{i})}\right)\qquad\text{et}\qquad
 \left(x_{i},-\frac{v_{Q}(x_{i})}{P_{Q}(x_{i})}\right)
$$
cependant les points $\{ \left(x_{i},-\frac{v_{Q}(x_{i})}{P_{Q}(x_{i})}\right) \}_{{1}\leqslant i \leqslant{g-2j}}$, alors il ne reste des zeros de $u_{Q^2} $ que les points $\{\left(x_{i},\frac{v_{Q}(x_{i})}{P_{Q}(x_{i})}\right) \}_{{1}\leqslant i \leqslant{g-2j}}$ qui sont des pôles de la fonction $F$.
Afin d'établir tous les pôles de $F$, on attribue $t$ comme paramètre local au voisinage de $\infty$, de telle sorte que $x=1/t^2$. Alors au voisinage de $\infty$, on a:
\begin{equation*}
 R(x)\sim \frac1{t^{2(k+n)}}\;,\quad P(x)\sim \frac1{t^{2n}}\;,\quad z\sim \frac1{t^{2g'+1}}\;,\quad {u(x)}\sim \frac1{t^{2(g'+n)}}. 
\end{equation*}%
Après substitution dans la fonction $F$ on trouve qu'au voisinage de $\infty$:
\begin{equation*}
F(x,z)\sim \frac1{t^{2(k+n)+1}}\;.
\end{equation*}%
Il résulte que les pôles de $ F$ sont $\left[(2(k+n)+1)\infty+\sum\limits_{i=1}^{g-2j} \left(x_{i},\frac{v_{Q}(x_{i})}{P_{Q}(x_{i})}\right)\right] $.
Du fait que le degré de tout diviseur principal est zéro, le diviseur $D$ est effectif et de degré $ g+2(n+k=j)+1$.

\smallskip

Le diviseur $D$ est étranger à $\supp{\mathfrak{m}}$, car la fraction $F$ ne s'annule pas sur le support de ${\mathfrak{m}} $, étant donné que $R_Q= \prod\limits_{i=1}^{k}(x-a_{i})P_Q$ s'annule en chaque $a_i$.
\end{proof}
\begin{rem}\label{jacrem}
Soit $M_{g,g}^0(h)$ l'ensemble des matrices $A(x)=\left(\begin{array}{cc}v(x)&{u(x)}\\w(x)&-v(x)\end{array}\right)\in M_{g,g}(h)$, tel que
${\PGCD(P,u,v)=1}$ est un ouvert dense de $M_{g,g}(h)$ où:
\begin{equation*}
 \frac{R(x)(P(x)z+v(x))}{u(x)}\equiv0\mod[\mathfrak{m}]
\end{equation*}%
Par consequent, la fonction $F$ nous permet d'avoir une $\mathfrak{m}$-équivalence entre ses zéros et ses pôles. Or, comme nous l'avons vu, les
pôles sont simples à écrire, donc sur l'ouvert $M_{g,g}^0(h)$ le morphisme $\Phi$ est le suivant:
\begin{equation}\label{div_}
 \left(\begin{array}{cc}
 v(x)&u(x)\\
 w(x)&-v(x)
 \end{array}\right)
 \longmapsto\theta\left(\sum\limits_{i=1}^{g}\left(x_{i},\frac{v(x_{i})} {P(x_{i})}\right)\right)\;.
\end{equation}%
le morphisme $\Phi$ généralise donc le morphisme de Mumford.
\end{rem} 
\begin{rem}
L'application $\Phi:M_{g,g}(h)\longrightarrow \Jac_{\mathfrak{m}}(C')$ est continue. 
%
\end{rem}

\smallskip

Afin de s'assurer que le morphisme $\Phi$ est le morphisme idéal pour décrire les strates du systeme de Mumford, il suffit de démontrer que $\Phi$ linéarise les champs de vecteurs du système de Mumford, à cet égard nous montrons que
ces champs deviennent des champs invariants par translation sur la jacobienne généralisée.\\
 
 Rappelons la relation entre les formes différentielles invariantes sur la jacobienne $\Jac_{\mathfrak{m}}(C')$ et les formes
différentielles de la courbe $C'$:
\begin{prop}
L'application $\theta^{*}$ est une bijection entre l'ensemble des formes différentielles invariantes de
$\Jac_{\mathfrak{m}}(C')$ et l'ensemble des formes différentielles rationnelles $\alpha$ de la courbe $C'$ telle que
$(\alpha)\geqslant -\mathfrak{m}$.
\end{prop} 

\smallskip

Nous commençons avec le plus simple des champs de vecteurs des systèmes de Mumford, le champ de vecteurs
$D_{g-1}$ qui, comme les autres champs de vecteurs $D_i$ est tangent à la strate $M_{g,g}(h)$.

\begin{prop}
Le champ de vecteurs $D_{g-1}$, restreint à la strate $M_{g,g}(h)$ est en dualité avec une forme différentielle
invariante de $\Jac_\mathfrak{m}({C'})$. Il est donc envoyé par $\Phi$ sur un champ de vecteurs invariant.
\end{prop}
\begin{proof}[Preuve]
Il suffit de restreindre la preuve à l'ouvert $M_{g,g}(h)^0$ dense de $M_{g,g}(h)$. Soit $A_0= \left(\begin{array}{cc}
 v_0&u_0\\
 w_0&-v_0
 \end{array}\right) \in M_{g,g}(h)^0$ où 
 $u^0(x)=\prod\limits_{i=1}^{g}(x-x^0_i)$ avec les $ x^0_i \in \mathbb{C}$. On sait que
\begin{align}
D_{g-1}\mid_{A^0}(u^0(x))=2v^0(x) \label{**}
\end{align}
Puisqu'il n'y a pas de confusion possible, et afin de simplifier les formules nous omettrons de mettre l'indice
$0$ et écrirons l'équation (\ref{**}) de la manière suivante
$$\displaystyle{-\sum\limits_{i=1}^{g} \prod\limits_{k \neq i}(x-x_k)\frac{dx_i}{dt}=2v(x).}$$
En évaluant cette dernière équation en $x_i$ avec $ 1 \leqslant i \leqslant g+n $, on obtient \\ 
$$\displaystyle{- \prod\limits_{k \neq i}(x_i-x_k)\frac{dx_i}{dt}=2v(x_i)}\;.$$ \\
Comme $v(x_i)=P(x_i)z_i$ où $(x_i,z_i) \in C'$ on a 
\begin{align}
- \prod\limits_{k \neq i}(x_i-x_k)\frac{dx_i}{dt} & =2P(x_i)z_i, \nonumber\\
\frac{dx_i}{P(x_i)z_i}& =-2\frac{dt}{ \prod\limits_{k \neq i}(x_i-x_k)} \label{s},
\end{align}
En sommant les équations (\ref{s}) pour tout $i$ on aura 
$$\sum\limits_{i=1}^{g} \frac{dx_i}{P(x_i)z_i}=0.$$
Rappelons que le polynôme $P(x)=\prod\limits_{j=1}^ k (x-a_j)^{l_j}$ est le diviseur quadratique maximal de $h$ et il est de degré $n$. En multipliant l'équation (\ref{s}) par les diviseurs du polynôme $P(x)$ suivant: $[(x-a_1), \dots, (x-a_1)^{l_1}, (x-a_1)^{l_1}(x-a_2), \dots, P(x)]$ et puis en évaluant au point $x_i$, on obtient: 
\begin{align}\label{sum}\left\lbrace\begin{array}{l}
\displaystyle{ \frac{dx_i}{(x_i-a_1)^{l_1-1}\prod\limits_{j\geqslant 2}^ k (x_i-a_j)^{\ell_j}z_i}}=-2(x_{i}-a_{1})\frac{dt}{ \prod\limits_{k \neq i}(x_i-x_k)}\;, \\ 
\displaystyle{ \frac{dx_i}{(x_i-a_1)^{\ell_1-2}\prod\limits_{j\geqslant 2}^ k (x_i-a_j)^{\ell_j}z_i}}=-2(x_{i}-a_{1})^{2}\frac{dt}{ \prod\limits_{k \neq i}(x_i-x_k)}\;, \\ 
\vdots \\ 
\displaystyle{ \frac{dx_i}{\prod\limits_{j\geqslant 2}^ k (x_i-a_j)^{\ell_j}z_i}}=-2(x_{i}-a_{1})^{\ell_{i}}\frac{dt}{ \prod\limits_{k \neq i}(x_i-x_k)}\;, \\ 
\displaystyle{ \frac{dx_i}{(x_i-a_2)^{\ell_2-1}\prod\limits_{j\geqslant 3}^ k (x_i-a_j)^{\ell_j}z_i}}=-2(x_{i}-a_{2})(x_{i}-a_{1})^{\ell_{i}}\frac{dt}{ \prod\limits_{k \neq i}(x_i-x_k)}\;, \\ 
\vdots \\
\displaystyle{ \frac{dx_i}{z_i}}=P(a_{i})\frac{dt}{ \prod\limits_{k \neq i}(x_i-x_k)}. \\ 
\end{array} \right. \end{align}
En sommant chaque équation de (\ref{sum}), on obtient
$$\left\lbrace\begin{array}{l}
\displaystyle{\sum\limits_{i=1}^{g+n} \frac{dx_i}{(x_i-a_1)^{\ell_1-1}\prod\limits_{j\geqslant 2}^ k (x_i-a_j)^{l_j}z_i}}=0, \\ 
\displaystyle{ \sum\limits_{i=1}^{g+n} \frac{dx_i}{(x_i-a_1)^{\ell_1-2}\prod\limits_{j\geqslant 2}^ k (x_i-a_j)^{l_j}z_i}}=0, \\ 
\vdots \\ 
\displaystyle{ \sum\limits_{i=1}^{g+n} \frac{dx_i}{\prod\limits_{j\geqslant 2}^ k (x_i-a_j)^{l_j}z_i}}=0, \\ 
\displaystyle{ \sum\limits_{i=1}^{g+n} \frac{dx_i}{(x_i-a_2)^{\ell_2-1}\prod\limits_{j\geqslant 3}^ k (x_i-a_j)^{l_j}z_i}}=0, \\ 
\vdots \\
\displaystyle{ \sum\limits_{i=1}^{g+n} \frac{dx_i}{z_i}}=0. \\ 
\end{array} \right. $$
Refaisons les mêmes étapes que précédemment en multipliant l'équation (\ref{s}) par $x^{{j}} P(x)$, pour $1 \leqslant {{j}} \leqslant g'-1$, puis évaluons l'équation obtenue au point $x_i$, et enfin, additionnons les équations pour tout $ 1\leqslant i \leqslant g$, on aura
$$\left\lbrace\begin{array}{l}
\displaystyle{\sum\limits_{i=1}^{g} \frac{x_idx_i}{z_i}}=0, \\ 
\displaystyle{\sum\limits_{i=1}^{g} \frac{x^2_idx_i}{z_i}}=0, \\ 
\vdots\\
\displaystyle{\sum\limits_{i=1}^{g} \frac{x^{g'-2}_idx_i}{z_i}}=0, \\
\displaystyle{\sum\limits_{i=1}^{g} \frac{x^{g'-1}_idx_i}{z_i}}=(-1)^{g+n-1}2dt. \\
\end{array} \right. $$
D'après Serre \cite{Ser},
$$ \left\langle \displaystyle{\frac{dx}{P(x)z}},\displaystyle{\frac{dx}{(x-{a_1})^{l_1-1}\prod\limits_{j\geqslant
 2}^k (x-a_j)^{l_j}z}},\displaystyle{\frac{dx}{(x-{a_1})^{l_1-2}\prod\limits_{j\geqslant 2}^ k
 (x-a_j)^{l_j}z}}, \dots,\displaystyle{\frac{dx}{z}},\displaystyle{\frac{xdx}{
 z}},\displaystyle{\frac{x^{g-1}dx}{ z}} \right\rangle$$
est l'espace vectoriel de formes différentielles rationnelles invariantes sur $C'$, qui définit les formes
différentielles invariantes sur $\Jac_\mathfrak{m}({C'})$. Le calcul ci-dessus montre alors que le champ de
vecteurs $D_{g-1}$ est annulé par toutes ces formes différentielles sauf par une forme pour laquelle il est en dualité. Cela prouve que le champ $D_{g-1}$ est envoyé par $\Phi$ sur un champs de vecteurs invariant par
translation.
\end{proof}

\begin{prop}\label{prp:lin}
Les champs de vecteurs $D^{g}_{g-1},\cdots, D^{g}_{0}$, restreints à la strate $M_{g,g}(h)$ correspondent à des
champs invariants sur la jacobienne généralisée $\Jac_\mathfrak{m}({C'})$.
\end{prop}
\begin{proof}[Preuve]
Comme les champs de vecteurs $D^{g}_{g-1},\cdots, D^{g}_{0}$ commutent, alors il est de même pour leur image. De
plus comme l'image de $D_{g-1}$ est invariante par translation et comme elle commute avec les images de
$D^{g}_{g-2},\cdots, D^{g}_{0}$, on conclut que les images de $D^{g}_{g-2},\cdots, D^{g}_{0}$ sont invariantes par
translation.
\end{proof}
\begin{prop}\label{fin}
 L'application $\Phi$ définit un isomorphisme entre $M_{g,g}(h)$ et un ouvert de la jacobienne généralisée
 $\Jac_{\mathfrak{m}}(C')$.
\end{prop}
\begin{proof}[Preuve]
Nous montrons d'abord que $\Phi$ est injective. Tout d'abord, $\Phi$ est localement injective, car d'une part les
deux variétés lisses $M_{g,g}(h)$ et $\Jac_{\mathfrak{m}}(C')$ ont la même dimension $g$ et de plus la proposition \ref{prp:lin} révèle que $\Phi$
envoie les $g$ champs de vecteurs indépendants $D_i$ sur $g$ champs de
vecteurs indépendants sur la jacobienne généralisée. En outre $\Phi$ est injective quand elle est restreinte à l'ouvert
dense $M_{g,g}^0(h)$ de $M_{g,g}(h)$:%
\begin{equation*}
 \left(\begin{array}{cc}
 v(x)&u(x)\\
 w(x)&-v(x)
 \end{array}\right)
 \longmapsto\theta\left(\sum\limits_{i=1}^{g}\left(x_{i},\frac{v(x_{i})} {P(x_{i})}\right)\right)\;.
\end{equation*}%
Sur cet ouvert, l'application qui associe à une matrice $A(x)$ un diviseur sur $C'$ est injective. De plus, Serre
montre que si deux diviseurs effectifs \emph{génériques} de degré $g$ sont $\mathfrak{m}$-équivalents, alors ils sont
égaux. Cela veut dire que $\Phi$ est à la fois localement injective et injective sur un ouvert dense. $\Phi$ est
donc injective et son image est un ouvert de $\Jac_{\mathfrak{m}}(C')$.
\end{proof}

\bibliographystyle{plain}

\end{document}